\documentclass[oneside,11pt,a4paper]{amsart}

\usepackage{amsmath,yhmath}
\usepackage{amssymb}
\usepackage{amsfonts}
\usepackage{bbm} %mathbbm
\usepackage{mathtools}
\usepackage{tikz}
\usepackage{tikz-cd}
\usepackage[inline]{enumitem}
\usepackage[mathscr]{eucal} % Euscript: \mathscr{}
\allowdisplaybreaks
\usepackage{stmaryrd} %widetriangle

\numberwithin{equation}{section}
\usepackage{caption}
\usepackage[margin=1in,marginparwidth=0.8in, marginparsep=0.1in]{geometry}
\usepackage{tikz}
\usepackage{aliascnt}
\usepackage[colorlinks, linkcolor=black,citecolor=blue, urlcolor=cyan, unicode, bookmarksnumbered]{hyperref} 
\newcommand*{\doi}[1]{\href{https://dx.doi.org/#1}{doi: #1}}

\newcommand*{\eeqref}[2][Equation~]{%
  \hyperref[{#2}]{#1(\ref*{#2})}%
}

\newtheoremstyle{note}
{\topsep}% above space
{\topsep}% below space
{}% body font
{0pt}% indent amount
{\bfseries}% head font
{.}% post head punctuation
{0.5em }% post head punctuation
{}% head spec
\theoremstyle{plain}

\newtheorem{Thm}{Theorem}[section]

\newtheorem*{Thm*}{Theorem}

\newaliascnt{prop}{Thm}
\newtheorem{Prop}[prop]{Proposition}
\aliascntresetthe{prop}

\newaliascnt{lemma}{Thm}
\newtheorem{Lemma}[lemma]{Lemma}
\aliascntresetthe{lemma}

\newaliascnt{coro}{Thm}
\newtheorem{Coro}[coro]{Corollary}
\aliascntresetthe{coro}

\newaliascnt{conjecture}{Thm}

\aliascntresetthe{conjecture}

\newtheorem{Lemma*}{Lemma}
\newtheorem*{Prop*}{Proposition}
\newtheorem*{Coro*}{Corollary}

\theoremstyle{definition}

\newaliascnt{def}{Thm}
\newtheorem{Def}[def]{Definition}
\aliascntresetthe{def}

\newtheorem*{Def*}{Definition}

\newaliascnt{Eg}{Thm}
\newtheorem{eg}[Eg]{Example}
\aliascntresetthe{Eg}

\newtheorem*{eg*}{Example}

\theoremstyle{remark}

\newaliascnt{rmk}{Thm}
\newtheorem{RMK}[rmk]{Remark}
\aliascntresetthe{rmk}

\newtheorem*{RMK*}{Remark}

\theoremstyle{plain}

\newtheorem{iThm}{Theorem}[section]

\newtheorem*{iThm*}{Theorem}

\newaliascnt{iprop}{iThm}

\aliascntresetthe{iprop}

\newaliascnt{iCoro}{iThm}

\aliascntresetthe{iCoro}

\newaliascnt{iconjecture}{iThm}

\aliascntresetthe{iconjecture}

\theoremstyle{definition}

\newaliascnt{idef}{iThm}

\aliascntresetthe{idef}

\newtheorem*{iDef*}{Definition}

\newaliascnt{iEg}{iThm}

\aliascntresetthe{iEg}

\newtheorem*{ieg*}{Example}

\theoremstyle{remark}

\newaliascnt{irmk}{iThm}

\aliascntresetthe{irmk}

\newtheorem*{iRMK*}{Remark}

\newcommand{\bR} { {\mathbb{R}}}

\newcommand{\bN} { {\mathbb{N}}}

\newcommand{\bfk} { {\mathbf{k}}}

\newcommand{\cC} { {\mathcal{C}}}
\newcommand{\cD} { {\mathcal{D}}}

\newcommand{\sT}{ \mathscr{T} }
\newcommand{\supp}{{\textnormal{supp}}}

\newcommand{\HOM}{{ \textnormal{Hom}  }}
\newcommand{\HHOM}{{ \mathcal{H}om  }}

\newcommand{\pt}{{\textnormal{pt}}}

\newcommand{\id}{{\textnormal{id}}}
\newcommand{\Mod}{{\textnormal{Mod}}}

\newcommand{\Sh}{{\operatorname{Sh}}}
\newcommand{\Fun}{{\operatorname{Fun}}}
\newcommand{\PrSt}{{\operatorname{Pr_{st}^L}}}

\newcommand{\cofib}[1]{{\operatorname{cofib}\left( #1 \right)}}
\title{Non-linear microlocal cut-off functors}
\author{Bingyu Zhang}
\date{\today}

\begin{document}

\begin{abstract}
To any conic closed set of a cotangent bundle, one can associate four functors on the category of sheaves, which are called non-linear microlocal cut-off functors. Here we explain their relation
with the microlocal cut-off functor defined by Kashiwara and Schapira, and prove a microlocal cut-off lemma for non-linear microlocal cut-off functors, adapting inputs from symplectic geometry. We also prove two K\"unneth formulas and a functor classification result for categories of sheaves with microsupport conditions.
\end{abstract}

\maketitle

\section{Introduction}
In the study of microlocal theory of sheaves, a crucial tool is the microlocal cut-off lemma of Kashiwara and Schapira (see \cite[Proposition 5.2.3, Lemma 6.1.5]{KS90}, \cite{D_Agnolo_cutoff} and \cite[Chapter III]{guillermou2019sheaves}), which constructs certain functors that enable us to cut off the microsupport of sheaves functorially. On the other hand, in the study of symplectic topology, very similar functors were defined and studied for different purposes \cite{tamarkin2013,NS20Weinstein, kuo2021wrapped}.

Precisely, for a smooth manifold $M$, a conic closed set $Z \subset T^*M$, and $U = T^*M \setminus Z$, we will study the \textit{non-linear microlocal cut-off functors} $L_U, L_{Z}, R_{U}, R_{Z} : \Sh(M) \rightarrow \Sh(M)$ together with fiber sequences of functors $L_U \rightarrow \mathrm{id} \rightarrow L_Z, , R_Z \rightarrow \mathrm{id} \rightarrow R_U$ defined using adjoint data of the split Verdier sequence $\Sh_Z(M) \rightarrow \Sh(M) \rightarrow \Sh(M;U)$. See \autoref{section: verdier seq from microlocalization} for more details.

The main goal of this article is proving a non-linear microlocal cut-off lemma using the wrapping formula of $L_Z$ introduced in \cite{kuo2021wrapped}. 

\begin{iThm}[{Non-linear microlocal cut-off lemma, \autoref{prop: non-linear microlocal cut-off lemma} below.}]\label{intro thm: non-linear microlocal cut-off lemma}
For a conic closed set $Z \subset T^*M$, and $U = T^*M \setminus Z$, we have

\begin{enumerate}
\item 
\begin{enumerate}
\item The morphisms $F\rightarrow L_Z(F)$ and $R_Z(F)\rightarrow F$ are isomorphisms if and only if $SS(F)\subset Z$.
\item The morphism $L_U(F) \rightarrow F$ is an isomorphism if and only if $F \in {}^{\perp}\Sh_Z(M)$, and the morphism $F \rightarrow R_U(F)$ is an isomorphism if and only if $F \in \Sh_Z(M)^{\perp}$.
\end{enumerate}
\item 
\begin{enumerate}

\item The morphisms $L_U(F)\rightarrow F$ and $F\rightarrow R_U(F)$ are isomorphisms on $U$.
\item If $0_M\subset Z$, the morphism $F\rightarrow L_Z(F)$ and $R_Z(F)\rightarrow F$ are isomorphisms on $\operatorname{Int}(Z)\setminus 0_M$. Equivalently, we have $SS(L_U(F))\cup SS(R_U(F)) \subset \overline{U}\cup 0_M$.

\end{enumerate}
\end{enumerate}
\end{iThm}

In fact, all of them except (2)-(b) follow directly from the definition. Only (2)-(b) of the theorem is a non-trivial fact.

Here, we shall explain the relation between non-linear microlocal cut-off functors and microlocal cut-off functors in \cite{KS90}. Here, we follow the notation of \cite[Section III.1]{guillermou2019sheaves} for microlocal cut-off functors.

We take $M=V$ to be a real vector space, $\lambda ,\gamma \subset V$ as pointed closed convex cones (we assume further that $\lambda$ is proper). The comparison is summarized in \autoref{table: functor comparision table}, where the first column marks corresponding references, and the last column marks their position in this article. Functors in the same columns are isomorphic with corresponding $U$ and $Z$.

\begin{table}[htbp]
\begin{tabular}{c|c|c|c|c|c|c|c}
\hline
This article &$Z$ & $U$ &$L_U$ &$L_Z$ & $R_U$ & $R_Z$ &\autoref{Definition: Non-linaer microlocal cut-off functor} \\ \hline
\cite[Section III.1]{guillermou2019sheaves}& $V\times \lambda^{\circ a}$ & &  $P_\lambda'$   &  $Q_\lambda$    &  $Q'_\lambda $   &  $P_\lambda    $ & \autoref{prop: functor comparision}\\ \hline
\cite[Section 3]{GS2014} &    &  $V\times \operatorname{Int}(\gamma^\circ)$ & $ L_\gamma$   &   &  $ R_\gamma$   & &\autoref{example: Tamarkin projectors} \\ \hline 
\end{tabular}
\caption{Known microlocal cut-off functors.}
\label{table: functor comparision table}
\end{table}

Based on this comparison, we see the relation between linear microlocal cut-off lemmas \cite[Proposition 5.2.3, Lemma 6.1.5]{KS90}, \cite[Proposition 3.17, 3.19, 3.20, 3.21]{GS2014} and our result. However, let me emphasize that our result is not completely irrelevant with those linear cut-off lemmas: In fact, to derive the wrapping formula of $L_Z$, we somehow use the 1-dimensional linear cut-off lemma for $\lambda=(-\infty,0]$. Therefore, one could say that \autoref{intro thm: non-linear microlocal cut-off lemma} constructs non-linear microlocal cut-off
functors, generalizing the linear microlocal cut-off functor which used the vector space structures, but it also provides a unified statement
for all microlocal cut-off functors. For example, we will explain how to prove a cut-off lemmas for the Tamarkin categories and its equivariant version in \autoref{example: Tamarkin category}.

We remark that our approach aims to explain the microlocal cut-off lemmas as globally as possible on $T^*M$ using input from symplectic topology. Therefore, we did not consider the refined microlocal cut-off lemmas (see \cite[Proposition 6.1.4 \& 6.1.6]{KS90}, \cite{D_Agnolo_cutoff}, and \cite[Section III.2, III.3]{guillermou2019sheaves}), which concern more on the local behavior of cut-off functors.

We also prove two K\"unneth formulas in \autoref{section: functor classification} (with more precise statements), which remove the isotropic condition of \cite[Theorem 1.1, 1.2]{Kuo-Li-Duality2024} via a different approach.
\begin{iThm}For two manifolds $M, N$, conic closed sets $Z\subset T^*M$ and $X\subset T^*N$, and we set $U=T^*M\setminus Z$ and $V=T^*N\setminus X$. We have
\[\Sh_{X\times Z}(N\times M) \simeq \Sh_{X  }(N )\otimes \Sh_{  Z}(  M),\quad \Sh(N;V)\otimes\Sh(M;U)  \simeq \Sh(N\times M; V\times U).\]
\end{iThm}
Combine the dualizability of $\Sh_Z(M)$ and $\Sh(M;U)$ via \cite[Remark 3.7]{Hochschild-Kuo-Shende-Zhang}, we have the following functor classification result:
\begin{align}\label{equation: kernel-functor-microlocalization}
\begin{aligned}
    &\Sh_{-X\times Z}(N\times M)\simeq \Fun^L(\Sh_X(N), \Sh_Z(M)) \simeq \Fun^R(\Sh_Z(M),\Sh_X(N) ) ^{op},\\
    &\Sh(N\times M; -V\times U)\simeq \Fun^L(\Sh(N;V), \Sh(M;U)) \simeq \Fun^R(\Sh(M;U), \Sh(N;V)) ^{op}.
\end{aligned}
\end{align}

\subsubsection*{Category convention}
In this article, a category means an $\infty$-category. We refer to \cite[Chapter 1]{Gaitsgory-Rozenblyum} for basics about higher algebra, and more details could be found in \cite{HTT,HA,SAG}.

In this article, we shall consider a fixed presentable stable symmetric monoidal category $(\bfk,\otimes,1)$, and we denote $\PrSt$ the category of $\bfk$-linear presentable stable categories and left adjoint $\bfk$-linear functors. There exists a $\bfk$-relative Lurie tensor product that form a closed symmetric monoidal structure on $\PrSt$ such that the category of left adjoint functors $\Fun^L$ serves as the internal hom, where we write $\otimes$/$\Fun^L$ directly instead of $\otimes_{\bfk}$/$\Fun^L_{\bfk}$. We say a $\bfk$-linear presentable stable categories is dualizable if it is a dualizable object in $\PrSt$.

For the purpose of this paper, we will restrict $(\bfk,\otimes,1)$ to be compactly generated, and be locally rigid in the sense of \cite[Definition 4.2.1, Corollary 4.3.5]{Fake_sheaves_on_mfd}. We refer to \cite[Section 4.2, 4.3, 4.4]{Fake_sheaves_on_mfd} and \cite{Ramzi_locally_rigid_cat} for more discussion on local rigidity. The standard example for $(\bfk,\otimes,1)$ would be the category of modules $\Mod_E$ over an $\mathbb E_\infty$ ring spectrum $E$. The compactly generation condition guarantees the microlocal sheaf theory with $\bfk$-coefficient to run without any modification as \cite{KS90}. The local rigidity condition ensures that plain adjoints of $\bfk$-linear functors are also $\bfk$-linear (\cite[Lemma 4.3.7]{Fake_sheaves_on_mfd}). We will discuss later in more detail the dependence of our results on the requirement. But let me remark here that: {\fontencoding{U}\fontfamily{futs}\selectfont\char 49\relax} Upon we can prove \cite[5.4.13, Theorem 7.2.1]{KS90} using the so-called $\Omega$-lens definition (i.e. the equivalent term (2), (3) of \autoref{lemma: Omega-lens}), then all the results of this paper are true if $\bfk$ is locally rigid but not necessarily compactly generated.

 \subsection*{Acknowledgements}This work is inspired by a discussion with St\'ephane Guillermou and Vivek Shende's interests on microlocal cut-off functors. The author thanks them for their helpful discussion and encouragement. We also thank for Alexander I. Efimov, Xin Jin, Christopher Kuo, Wenyuan Li, Pierre Schapira and Matthias Scharitzer for helpful discussion. Thanks for the anonymous referees and their helpful comments. This work was supported by the Novo Nordisk Foundation grant NNF20OC0066298 and VILLUM FONDEN, VILLUM Investigator grant 37814.

\section{Sheaves and integral kernel}\label{section: sheaf theory review}
For a locally compact Hausdorff topological space $X$, we set $\Sh(X;\cC)$ to be the category of $\cC$-valued sheaves \cite[7.3.3.1]{HTT} for a presentable symmetric monoidal category, which is also $\cC$-linear. It is explained in \cite{6functor-infinity} that since $\cC$ is a symmetric monoidal category, we can define the $\cC$-linear $6$-functor formalism. We also refer to \cite{Six-FunctorScholze} for the $6$-functor formalism. In this paper, we mostly interested in the case that $\cC=\bfk$, and in this case, we simply denote $\Sh(X)=\Sh(X;\bfk)$. In this case, it is explained in \cite[Lemma 4.6.17.]{Fake_sheaves_on_mfd} that local rigidity can guarantee the Verdier dual functor behave as it is expected.

For a locally closed inclusion $i:Z\subset X$ and $F\in \Sh(X)$, we set $F_Z = i_!i^{-1}F$ and $\Gamma_Z F = i_*i^!F$. We denote the constant sheaf $1_X = a^*1$, where $a$ is the constant map $X\rightarrow \pt$, and denote $1_Z = (1_X)_Z\in \Sh(X)$ if there is no confusion.

The main reason for using $\infty$-category is the following classification of left adjoint functors, which is a combination of a sequence of Lurie's results. We refer to {\cite[Proposition 3.2]{Hochschild-Kuo-Shende-Zhang}} for a proof.
\begin{Prop}\label{Kernel-functor corresponding}If $H_1$ is a locally compact Hausdorff space, then for all topological spaces $H_2$, we have the equivalence of categories
\[\Sh(H_1 \times H_2) \simeq \Fun^L(\Sh(H_1), \Sh(H_2)) \]
that sends $K \in \Sh(H_1 \times H_2)$ to the convolution functor $\Phi_K=[F\mapsto p_{2!}(K\otimes p_1^*F)]$, where $p_i:H_1\times H_2 \rightarrow H_i$ are projections. We call $K$ the kernel of $\Phi_K$.
\end{Prop}

It is known that taking the right adjoint induces an equivalence of categories
\[\Fun^L(\Sh(H_1), \Sh(H_2)) \simeq \Fun^R(\Sh(H_2), \Sh(H_1))^{op},\]
where $\Fun^R$ stands for right adjoint functors.

However, for any convolution functor $\Phi_K: \Sh(H_1)\rightarrow \Sh(H_2)$ with $K\in \Sh(H_1 \times H_2)$, there exists an obvious right adjoint functor\footnote{We call it an nvolution functor since it is a dual of a convolution functor, while coconvolution should be the same with nvolution.}
\[\Psi_K:\Sh(H_2) \rightarrow \Sh(H_1),\, F\mapsto p_{1*}\HHOM(K,p_2^!F).\]

Then we have the following equivalence of categories
\begin{Coro}\label{Kernel-functor corresponding right}Under the same condition of \autoref{Kernel-functor corresponding}, we have
  \[\Sh(H_1 \times H_2) \simeq \Fun^R(\Sh(H_2), \Sh(H_1))^{op},\, K\mapsto \Psi_K.\]  
\end{Coro}
\begin{RMK}\autoref{Kernel-functor corresponding} and \autoref{Kernel-functor corresponding right} are only true on the $\infty$-category level. In the classical theory of microlocal sheaves, many functors are built as triangulated convolution (nvolution) functors, which do not admit similar results. Especially, we do not know if the kernel of triangulated convolution (nvolution) functors are unique in general.
\end{RMK}

\section{Microsupport of sheaves}\label{Section: microsupport}
From now on, we assume $M$ is a smooth manifold. Regarding the microlocal theory of sheaves in the $\infty$-categorical setup, we remark that all arguments of \cite{KS90} work well provided we have the non-characteristic deformation lemma \cite[Proposition 2.7.2]{KS90} for all sheaves. When $\bfk$ is compactly generated, the non-characteristic deformation lemma is proven for all hypersheaves (\cite{Amicrolocallemma_infinitycat}), and then for all sheaves because of hypercompleteness for manifolds (\cite[7.2.3.6, 7.2.1.12]{HTT}, and \cite[Section 1]{Haine-porta-Teyssier-homotopy-inv-constr}).

To any object $F \in \Sh(M)$, one can associate a conic closed set $SS(F) \subset T^*M$ ({\cite[Definition 5.1.2]{KS90}}), which satisfies the following triangle inequality: for a fiber sequence $F \rightarrow G \rightarrow H$, we have $SS(F) \subset SS(G) \cup SS(H)$. 

We recall the following definition and result regarding on an equivalent definition of microsupport.
\begin{Def}{\cite[Definition 3.1]{guillermouviterbo_gammasupport}}Let $\Omega \subset T^*M \setminus 0_M$ be an
open conic subset.  We call {\em $\Omega$-lens} a locally closed subset $\Sigma$ of $M$
with the following properties: $\overline{\Sigma}$ is compact and there exists an open
neighborhood $U$ of $\overline{\Sigma}$ and a function $g\colon U \times [0,1] \to \bR$
\begin{enumerate}
\item $dg_t(x) \in \Omega$ for all $(x,t) \in U \times [0,1]$, where
  $g_t = g|_{U\times\{t\}}$,
\item $\{g_t<0\} \subset \{g_{t'}<0\}$ if $t\leq t'$,
\item the hypersurfaces $\{g_t=0\}$ coincide on $U\setminus \overline{\Sigma}$,
\item $\Sigma = \{g_1<0\} \setminus \{g_0<0\}$.
\end{enumerate}
    
\end{Def}

\begin{Lemma}{\cite[Lemma 3.2, 3.3]{guillermouviterbo_gammasupport}}\label{lemma: Omega-lens} Let $F \in \Sh(M)$ and let $\Omega \subset T^*M \setminus 0_M$ be an open conic subset.  Then following are equivalent: 1) $SS(F) \cap \Omega = \emptyset$, 2) $\HOM(1_\Sigma, F) \simeq 0$ for any $\Omega$-lens $\Sigma$, 3) $\Gamma(M; 1_\Sigma \otimes F) \simeq 0$ for any $(-\Omega)$-lens $\Sigma$.
\end{Lemma}

\begin{RMK}\label{remark: compactly generated}The proof of this lemma needs the non-characteristic deformation lemma (\cite[Proposition 2.7.2]{KS90}, \cite{Amicrolocallemma_infinitycat}), which is not true if $\bfk$ is not compactly generated \cite[Remark 4.25]{Efimov-K-theory}.

The lemma motivates a new definition\footnote{The author learned this idea from St\'ephane Guillermou and Marco Volpe.} of microsupport using $\Omega$-lens and the equivalent
conditions (2) or (3), which then avoids the use of non-characteristic deformation. The definition works very well without $\bfk$ being compactly generated, and it is very likely equivalent to the one proposed in \cite[Remark 4.24]{Efimov-K-theory}. One can expect many results in microlocal sheaf theory are still true with this definition (in the case that $\bfk$ is not compactly generated). Nevertheless, we will not discuss further the extension of the definition of microsupport in this article, and restrict ourselves in the compactly generation case for safety.  
\end{RMK}

The lemma leads to the following property of microsupports, where the microlocal cut-off lemma in \cite[Section 5.2]{KS90} is not involve.
\begin{Prop}[{\cite[Exercise V.7]{KS90}}]\label{closure property of microsupport}For a set of sheaves $F_\alpha\in \Sh(M)$ indexed by $\alpha \in A$, we have 
\[SS(\prod_{\alpha}F_\alpha)\cup SS(\bigoplus_{\alpha}F_\alpha) \subset \overline{\bigcup_\alpha SS(F_\alpha)}.\]     
\end{Prop}

For $X_n\subset X$ with $n\in \bN$, we set
\begin{equation*}
    \begin{split}
        {\limsup_n} X_n &= \bigcap_{N\geq 1}\overline{\bigcup_{n\geq N} X_n}\\
       &= \{x: \exists (x_n) \text{ such that }x_n\in X_n\text{ for infinitely many }n, \, x_n\rightarrow x\},\\
        {\liminf_n} X_n &= \{x: \exists (x_n)\text{ such that }x_n\in X_n\text{ for all }n, \, x_n\rightarrow x\}.
    \end{split}
\end{equation*}
By definition, we have ${\liminf_n} X_n \subset {\limsup_n} X_n$. As a corollary of \autoref{closure property of microsupport}, we have
\begin{Coro}[{\cite[Proposition 6.26]{guillermouviterbo_gammasupport}}]\label{coro: colimit estimation}Denote $N(\bN)$ the nerve of the $1$-category $\bN$. For a functor $N(\bN)\rightarrow \Sh(X)$, we have
\[SS(\varinjlim_{n}F_n)  \subset   {\liminf_n}\, SS(F_n).\] 
    \end{Coro}
\begin{RMK}\label{remark: colimit using telescope}
We can also denote $\bN$ as a simplicial set which consists of all vertices together with the edges that join consecutive integers, and no higher simplexes. Then the natural inclusion of simplicial sets $\bN \rightarrow N(\bN)$ is cofinal. Therefore, we can compute the colimit using a mapping telescope construction as in triangulated categories, and the proof of \autoref{coro: colimit estimation} is the same as in the references.
\end{RMK} 
For our later application, we present its proof here.
\begin{proof}
By \autoref{remark: colimit using telescope}, and the fact that all strictly increasing sequences are cofinal in $\bN$, we have cofiber sequences for all strictly increasing sequences $\{n_i\}_{i \in \mathbb{N}}$,
\[ \bigoplus_{i}F_{n_i} \rightarrow \bigoplus_{i}F_{n_i} \rightarrow   \varinjlim_{n}F_n .\]
We first take $n_i=i+N$ for all $N\geq 1$, and then we have $SS(\varinjlim_{n}F_n) \subset  {\limsup_n} SS(F_n)$ by \autoref{closure property of microsupport} and the triangle inequality.

Now, take $x\notin {\liminf_n} SS(F_n)$. By the definition of $ {\liminf_n}  $, we can find a strictly increasing sequence $n_i$ such that $x \notin {\limsup_i} SS(F_{n_i})$, then we have $x \notin SS(\varinjlim_{i} F_{n_i}) = SS(\varinjlim_{n} F_{n})$. That is, $SS(\varinjlim_{n} F_{n}) \subset {\liminf_n} SS(F_n)$.
\end{proof}

Another application of \autoref{lemma: Omega-lens} is the following property, which could be understood as a version of $1$-dimensional microlocal cut-off lemma. We denote $i_s:M\to M\times \bR,x\mapsto (x,s)$. 
\begin{Prop}\label{prop: continuation map using from lens}For $F\in \Sh(M\times \bR)$, if $SS(F)\subset T^*M \times \bR\times (-\infty,0]$, then for $a\leq b$ there exists a morphism $c(a,b):i^*_aF \rightarrow i^*_bF$.
\end{Prop}
\begin{proof}Let $\Omega=T^*M\times \bR\times (0,\infty)$. For any relatively compact open set $U\subset M$, we take a $\Omega$-lens $\Sigma=U\times [y,z)$ for $y<z$. Therefore, by \autoref{lemma: Omega-lens}, the following morphism is naturally an equivalence for $x<y<z$:
\[F(U\times (x,z))\xrightarrow{\simeq} F(U\times (x,y)).\]
For general $U$, the equivalence is still true: We take an open cover $\{U_\alpha\}_\alpha$ of $U$, and then we use $F(U\times (y,z))=\varprojlim_\alpha F(U_\alpha\times (y,z))$ since $F$ is a sheaf.

Then we can define a morphism for open sets $U\subset M$, which is natural with respect to $U$:
\[F(U\times (a-\varepsilon,a+\varepsilon)) \simeq F(U\times (a-\varepsilon,b+\varepsilon)) \rightarrow F(U\times (b-\varepsilon,b+\varepsilon)).     \]

Pass to $\varinjlim_{\varepsilon \rightarrow 0}$, we define a morphism of presheaves 
\[c^{pre}(a,b):i^{*,pre}_aF \rightarrow i^{*,pre}_bF,\]
where we denote $i^{*,pre}_s$ the presheaf pullback. 

Then we define $c(a,b)$ as the sheafification of $c^{pre}(a,b)$.
\end{proof}

Lastly, we need the following microsupport estimation of (co)nvolution functors:
\begin{Prop}[{\cite[Lemma 4.4, Proposition 4.5]{Kuo-Li-Duality2024}}]\label{prop: convolution estimation in general}Let $K\in \Sh(N\times M)$ with $SS(K)\subset (-X)\times Z$ for conic closed sets $X\subset T^*N$, $Z\subset T^*M$, then we have, for $F\in \Sh(N)$,
\[SS(\Phi_K(F)) \subset Z,\quad SS(\Psi_K(F))\subset X.\]
\end{Prop}
Let us also give a proof of this proposition using $\Omega$-lens.
\begin{proof}We assume that $X$ and $Z$ contain the
zero section, the general case can be proven using $SS(F)\cap 0_M=\operatorname{supp}(F)$. By \autoref{lemma: Omega-lens}, we only need to show that for a $-(T^*M\setminus Z)$-lens $\Sigma$, we have
\[\Gamma(M, 1_\Sigma \otimes \Phi_K(F))=0.\]
By definition of $\Omega$-lens, the closed set $\overline{\Sigma}$ is compact, then we can replace $\Gamma=p_{2*}$ by $\Gamma_c=p_{2!}$. Then we use base change formula and projection formula to see that we have
\[\Gamma(M, 1_\Sigma \otimes \Phi_K(F))=\Gamma_c(N\times M, (F\boxtimes 1_\Sigma )\otimes K).\]
Notice that $F=\varinjlim_\alpha 1_{W_\alpha}$ for relatively compact open sets $W_\alpha$ by \cite[Proposition 5.4.5]{Gaitsgory-Rozenblyum} and the definition of sheaves. We have
\[\Gamma(M, 1_\Sigma \otimes \Phi_K(F))=\varinjlim_\alpha \Gamma(N\times M, (1_{W_\alpha\times \Sigma} )\otimes K).\]  
Therefore, notice that 
\[W_\alpha\times \Sigma\]
is a ${T}^*N\times (-(T^*M\setminus Z))= - {T}^*N\times (T^*M\setminus Z)$-lens, we can conclude by \autoref{lemma: Omega-lens} that 
\[\Gamma(N\times M, (1_{W_\alpha\times \Sigma} )\otimes K)=0\]
since $SS(K)\subset (-X)\times Z \subset T^*N\times Z$.

The statement about $\Psi_K(F)$ can be proven using \autoref{lemma: Omega-lens}-(2) and a similar argument. We leave the details to the readers.
\end{proof}

\section{Split Verdier sequence from microlocalization}\label{section: verdier seq from microlocalization}
Throughout this article, we set $Z\subset T^*M$ to be a conic closed set and $U=T^*M\setminus Z$ (which is a conic open set). We set $\Sh_Z(M)$ as the full subcategory of $\Sh(M)$ spanned by sheaves $F$ with $SS(F)\subset Z$.

\begin{Prop}\label{prop: refletive}The category $\Sh_Z(M)$ is a stable subcategory of $\Sh(M)$ closed under small limits and colimits.

In particular: 1) The inclusion $\iota:\Sh_Z(M) \rightarrow  \Sh(M) $ admits both left and right adjoints. 2) $\Sh_Z(M)$ is $\bfk$-linear. 
\end{Prop}
\begin{proof}
By the triangle inequality of microsupport, we only need to show that $\Sh_Z(M)$ is closed under small products and coproducts, which follows from \autoref{closure property of microsupport}.

The inclusion $\iota$ admits both adjoints by the adjoint functor theorem \cite[5.5.2.9]{HTT}; $\Sh_Z(M)$ is presentable and $\bfk$-linear by \cite{Reflection-Theorem-infty} since $\iota$ is reflective and $\bfk$ is locally rigid.
\end{proof}
\begin{RMK}\begin{enumerate}[fullwidth]
    \item This is the only place of the article that we need explicitly use the local rigidity assumption of $\bfk$.
    \item Another way to formulate \autoref{lemma: Omega-lens} is that, the left (right) semi-orthogonal complement $\Sh_Z(M)$ is generated by $1_\Sigma$ for $T^*M\setminus Z$-lens ($- (T^*M\setminus Z)$-lens under colimits (limits). This observation was also provided in \cite{JinTruemann2017}, where the condition (3) of \autoref{lemma: Omega-lens} was formulated in terms of dual of $1_\Sigma$. It does not make a significant difference. as we can assume the lens $\Sigma$ is small enough that we can orient a small neighborhood of $\Sigma$.
\end{enumerate}

\end{RMK}
\begin{Def}[{\cite[Definition 6.1.1]{KS90}}]For the conic open set $U\subset T^*M$, we set
\[\Sh(M;U)\coloneqq \Sh(M)/\Sh_Z(M).\]  
\end{Def}
We refer to \cite[Section 5]{universal-hihger-K}, \cite[Section 1.3]{Nikolaus-Scholze-Ontopologicalcyclichomology} and \cite[Appendix A]{Hermitian-K-theory} for basic results of Verdier quotients of stable categories. Here, we say that $\cC\xrightarrow{\iota}\cD \xrightarrow{j} \cD/\cC$ is a \textit{split Verdier sequence} if $\iota$ admits both left and right adjoints, see \cite[A.2.4, A.2.6]{Hermitian-K-theory}. We denote $^{\perp}\cC$ (resp. $\cC^{\perp}$) as the left (resp. right) semi-orthogonal complement of $\cC$ in $\cD$, and we can identify $^{\perp}\cC$ (resp. $\cC^{\perp}$) with $\cD/\cC$ via a left (resp. right) adjoint functor in the split case.

By \autoref{prop: refletive}, the inclusion $\iota: \Sh_Z(M) \rightarrow \Sh(M)$ defines a split Verdier sequence. Precisely, we have the diagram of functors:
\begin{equation}\label{equation: recollection diagram}
    \begin{tikzcd}
\Sh_Z(M)  \arrow[r, "\iota"] & \Sh(M)  \arrow[r, "j"] \arrow[l, "\iota^*"', bend right=30] \arrow[l, "\iota^!", bend left=30] & {\Sh(M;U).} \arrow[l, "j_!"', bend right=30] \arrow[l, "j_*", bend left=30]
\end{tikzcd}
\end{equation}
Then we have adjunction pairs
\[L_U\coloneqq j_!j \dashv j_*j \eqqcolon R_U,\qquad L_Z\coloneqq \iota\iota^* \dashv \iota\iota^! \eqqcolon R_Z,\]
and the units/counits give us the following fiber sequences of functors on $\Sh(M)$
\begin{equation}\label{equation: fiber sequence non-linear cut-off}L_U\rightarrow \id\rightarrow L_Z,\quad
        R_Z\rightarrow \id\rightarrow R_U.    
\end{equation}
By \autoref{Kernel-functor corresponding} and \autoref{Kernel-functor corresponding right}, there exists a fiber sequence in $\Sh(M\times M)$
\begin{equation}\label{equation: fiber sequence for kernel}
    K_U\rightarrow 1_{\Delta_M}\rightarrow K_Z
\end{equation}
that gives corresponding convolution/nvolution functors
\begin{equation}\label{equation: kernel for non-linear cut-off functors}
L_U=\Phi_{K_U},\,L_Z=\Phi_{K_Z},\,R_U=\Psi_{K_U},\,R_Z=\Psi_{K_Z}  
\end{equation}
and corresponding natural transformations. Therefore, in practice, to construct those functors and fiber sequences between them, we only need to write down the fiber sequence \eqref{equation: fiber sequence for kernel}.

\begin{Def}\label{Definition: Non-linaer microlocal cut-off functor}We call the functors $L_U, L_Z, R_U, R_Z$ the \textit{non-linear microlocal cut-off functors}. Corresponding kernels $K_U,K_Z$ are called microlocal cut-off kernels (or microlocal kernels for short).  
\end{Def}

For any object in $\Sh(M;U)$, the triangle inequality implies that $SS_U([F])\coloneqq SS(F)\cap U$ for any representative $F \in \Sh(M)$ is well-defined. Then we have $SS_U([F])=SS(R_U(F))\cap U =SS(L_U(F))\cap U$ for all $F\in\Sh(M)$. For conic closed sets $X\supset Z$, we consider the category $\frac{\Sh_{X}(M)}{\Sh_{Z}(M)} $ as a full subcategory of $ \Sh(M;U) $. Then the essential image of $\frac{\Sh_{X}(M)}{\Sh_{Z}(M)} $ in $ \Sh(M;U) $ is $ \Sh_{X\cap U}(M;U) $, the full subcategory spanned by objects with $SS_U([F])\subset X\cap U$. We will recall this construction in \autoref{section: pair}.

\begin{eg}\label{example: excision}
Take an open set $W\subset M$, and set $U=T^*W\subset T^*M$ with $Z=T^*M\setminus T^*W$. We naturally identify $\Sh_Z(M)$ with sheaves supported in $M\setminus W$ since $\pi_{T^*M}(SS(F))=\supp{F}$ for the cotangent projection $\pi_{T^*M}:T^*M\rightarrow M$, and we can verify that the restriction functor $(\bullet)|_W:\Sh(M)\rightarrow \Sh(W)$ exhibits $\Sh(W)$ as the Verdier quotient $\Sh(M;T^*W)$. Then we have
\[L_U(F)=F_W, \,L_Z(F)=F_{M\setminus W}, \, R_U(F)=\Gamma_{W}(F),\,R_Z(F)=\Gamma_{M\setminus W}(F), \]
and can take microlocal kernels as following:
\[K_U=1_{\Delta_W}\rightarrow 1_{\Delta_M}\rightarrow K_Z=1_{\Delta_{M\setminus W}}.\]
Therefore, the fiber sequences \eqref{equation: fiber sequence non-linear cut-off} are exactly the excision sequence for sheaves. This example also motivates the notation of functors in \eeqref{equation: recollection diagram}.
\end{eg}

In the end of this section, we state the main result of this article. Recall that we say a morphism $f:F\rightarrow G$ in $\Sh(M)$ is an isomorphism on an open set $U$ if $SS(\cofib{f})\cap U=\varnothing $, equivalently, $f$ is an isomorphism in $\Sh(M;U)$ (cf. \cite[Definition 6.1.1]{KS90}).
\begin{Thm}\label{prop: non-linear microlocal cut-off lemma}For a conic closed set $Z \subset T^*M$, and $U = T^*M \setminus Z$, we have

\begin{enumerate}
\item 
\begin{enumerate}
\item The morphisms $F\rightarrow L_Z(F)$ and $R_Z(F)\rightarrow F$ are isomorphisms if and only if $SS(F)\subset Z$.
\item The morphism $L_U(F) \rightarrow F$ is an isomorphism if and only if $F \in {}^{\perp}\Sh_Z(M)$, and the morphism $F \rightarrow R_U(F)$ is an isomorphism if and only if $F \in \Sh_Z(M)^{\perp}$.
\end{enumerate}
\item 
\begin{enumerate}

\item The morphisms $L_U(F)\rightarrow F$ and $F\rightarrow R_U(F)$ are isomorphisms on $U$.
\item If $0_M\subset Z$, the morphism $F\rightarrow L_Z(F)$ and $R_Z(F)\rightarrow F$ are isomorphisms on $\operatorname{Int}(Z)\setminus 0_M$. Equivalently, we have $SS(L_U(F))\cup SS(R_U(F)) \subset \overline{U}\cup 0_M$.

\end{enumerate}
\end{enumerate}
\end{Thm}
\begin{proof} (1)-(a) and (2)-(a) follows from definition of $L_Z$ and $L_U$ and fiber sequences \eqref{equation: fiber sequence non-linear cut-off}. (1)-(b) follows from \cite[A.2.8]{Hermitian-K-theory}. 
For (2)-(b), we will prove in \autoref{prop: microsupport estimation of kernel} that $SS(K_U) \subset  (-\overline{U})\times \overline{U} \cup 0_{M\times M}$ under the condition $0_M\subset Z$. Therefore, $SS(L_U(F))\cup SS(R_U(F))\subset \overline{U}  \cup 0_M $ follows from \autoref{prop: convolution estimation in general}.
\end{proof}

\section{Wrapping formula of non-linear microlocal cut-off functors}\label{section: wrapping formula}
In this section, we present an explicit formula of microlocal kernels using Guillermou-Kashiwara-Schapira sheaf quantization \cite{GKS2012}. We recall the results of \textit{ loc.cit.} here. 

Let $\dot{T}^*M$ be the complement of the zero section in $T^*M$, and, for subset $A\subset T^*M$, we set $\dot{A}=A \cap  \dot{T}^*M$. In particular, we have the notion of $\dot{SS}(F)$ for $F\in \Sh(M)$.

Let $(I, 0)$ be a pointed interval. Consider a $C^\infty$ conic symplectic isotopy
\[\phi: I \times \dot{T}^*M \rightarrow \dot{T}^*M\]
which is the identity at $0 \in  I$. Such an isotopy is always the Hamiltonian flow for a unique conic function $H:  I \times \dot{T}^*M  \to \bR$ and we set $\phi=\phi_H$ when emphasize the Hamiltonian functions. At fixed $z \in I$, we have the graph of $\phi_{z}$: 
\begin{equation}\label{equation: graph}
  \Lambda_{{\phi}_{z}}\coloneqq \left\lbrace ((q,-p),{\phi}_{{z}}(q,p) ) : (q,p) \in \dot{T}^*M\right \rbrace \subset \dot{T}^*M\times \dot{T}^*M\subset \dot{T}^*(M\times M) .
\end{equation}
As for any of Hamiltonian isotopy, we may consider the Lagrangian graph, which by definition is a Lagrangian subset $\Lambda_\phi \subset  T^*I \times  \dot{T}^*(M\times M) $
with the property that 
$ \Lambda_{{\phi}_{z_0}}$ is 
the symplectic reduction 
of $\Lambda_{{\phi}}$ along  $\{z=z_0\}$.  It is given by the formula: 

\begin{equation}\label{equation: totalgraph}
  \Lambda_{{\phi}}\coloneqq \left\lbrace (z, - H(z,{\phi}_{z}(q,p))  , (q,-p),{\phi}_{z}(q,p)) : z\in I, (q,p) \in \dot{T}^*M \right \rbrace 
\end{equation}

\begin{Thm}[{\cite[Theorem 3.7, Proposition 4.8]{GKS2012}}]\label{theorem: GKS} For $\phi$ as above, there is a sheaf $K=K({\phi}) \in \Sh( I\times  M^2)$ such that $\dot{SS}(K)\subset \Lambda_{{\phi}}$ and $K|_{ \{0\}\times M^2}\cong 1_{\Delta_{M}}$. The pair $(K, K|_0)\cong (K,1_{\Delta_{M}})$ is unique up to a unique isomorphism.

Moreover, for isotopies $\phi_H, \phi_{H'}$ with
$H' \leq H$, there's a map $K(\phi_{H'})|_1 \to K(\phi_H)|_1$.  
In particular, when $H \geq 0$, then there is a map $1_{\Delta_M} \to K(\phi_H)|_1$. 
\end{Thm}
\begin{RMK}\label{RMK: GKS over general base}\begin{enumerate}[fullwidth]
\item In the proof of \autoref{theorem: GKS}, requirements for $\bfk$ contributes to different part of the proof. 

As we explained in the beginning \autoref{section: sheaf theory review}, we need local rigidity if we want to use the Verdier dual. It basically shows up everywhere in the proof.

However, it is not completely clear that if the compact generation is essential. The main role of compact generation here is to guarantee that we can use arguments in \cite{KS90} to prove the following  type of microsupport estimations: Box product, pullback and pushforward. The proofs of the box product estimation and pullback by submersion estimation in \cite[Proposition 5.4.1, 5.4.5]{KS90} are essentially using the $\Omega$-lens definition; it is directly to prove the pushforward estimation \cite[Proposition 5.4.4]{KS90} using the $\Omega$-lens definition; but we do not know if pullback by embedding can be proven using the $\Omega$-lens definition directly. Upon one can do this, the GKS theorem is true in the case that $\bfk$ is locally rigid but not necessarily compactly generated. 

\item The existence does not use the linear cut-off lemmas. The continuation map can be derived directly from the \eeqref{equation: graph} and \autoref{prop: continuation map using from lens}, which is a version of 1-dimensional linear cut-off lemma, upon we have the existence. 
\end{enumerate}
    
\end{RMK}

Motivated by ideas of \cite{Nadler-pants, GPS3}, it was shown in \cite{kuo2021wrapped} that for any closed set $ Z^\infty \subset S^*M$ and the conic closed set $Z= \bR_{>0} Z^\infty\cup 0_M \subset T^*M$, the left and right adjoint inclusion $\iota: \Sh_Z(M) \rightarrow \Sh(M)$ can be computed `by wrapping'. More precisely, 
 \begin{Thm}[{\cite[Theorem 1.2]{kuo2021wrapped}}] \label{chris thesis}
If $H_n$\footnote{Here, we do not distinguish a function on $S^*M$ and its conic lifting on $\dot{T}^*M$.} is any increasing sequence of positive compactly supported Hamiltonians supported on $S^*M \setminus Z^\infty$ such that $H_n \uparrow \infty$ pointwise in $S^*M \setminus Z^\infty$, then the adjoints of $\iota$ could be computed by
\begin{align}
\iota^* F = \varinjlim \Phi_{K(\phi_{H_n})|_1 }( F), \quad  \iota^! F = \varprojlim \Phi_{K(\phi_{-H_n})|_1 }( F).
\end{align} 

Moreover, the continuation maps $F \rightarrow \Phi_{K(\phi_{H_n})|_1 }( F)$ (resp. $  \Phi_{K(\phi_{-H_n})|_1 }( F) \rightarrow F$) induced by $1_{\Delta_M} \to  K(\phi_{H_n})|_1$ (resp. $  K(\phi_{-H_n})|_1 \to 1_{\Delta_M} $ ) from positivity of $H_n$ give the unit (resp. counit) after taking colimit (resp. limit).\end{Thm}
\begin{RMK}
\begin{enumerate}[fullwidth] 
\item In fact, as explained in \cite[Remark 6.5]{Hochschild-Kuo-Shende-Zhang}, the colimit in \cite{kuo2021wrapped} is taken over an $\infty$-categorical `wrapping category', but one can compute the colimit by a cofinal sequence as explained in \cite[Lemma 3.31]{kuo2021wrapped}.

\item As we remark in \autoref{RMK: GKS over general base}, we may think the wrapping formula as a cut-off result derived from the 1-dimensional linear cut-off lemma.
\item In \cite{kuo2021wrapped}, the author require $\bfk$ to be compactly generated and rigid to ensure the present proofs for the existence and uniqueness parts of GKS theorem work. We have constructed the continuation map using the $\Omega$-lens definition, which only requires that $\bfk$ to be dualizable.

Upon we have the full GKS theorem, even the argument of \cite[Theorem 1.2]{kuo2021wrapped} was written using \cite[Definition 5.1.2]{KS90}, but can be written in terms of the $\Omega$-lens. 
   
\item It is clear that we also have $\iota^! F= \varprojlim \Psi_{K(\phi_{H_n})|_1 }( F)$ via nvolution functors. The formula coincides with the above formula using limit of convolutions since $\Psi_{K(\phi_{H_n})|_1 }\simeq \Phi_{K(\phi_{-H_n})|_1 }$ by \cite[Proposition 3.2(ii)]{GKS2012}.
\end{enumerate}

\end{RMK}
 
For $H\geq 0$, We set $K^\circ(\phi_{H_n})= \cofib{1_{I\times \Delta_M} \to  K(\phi_{H_n})}.   $

Compose with the natural inclusion $\iota$ and combine with \cite[Proposition 3.5]{kuo2021wrapped}, one can see that the fiber sequence \eqref{equation: fiber sequence for kernel} can be taken as
\begin{equation}\label{equation: wrapping formula}
   K_U=\varinjlim (K^\circ(\phi_{H_n})|_1 ) \rightarrow 1_{\Delta_M}  \rightarrow  K_Z=\varinjlim (K(\phi_{H_n})|_1 ).   
\end{equation}
 
To complete the proof of \autoref{prop: non-linear microlocal cut-off lemma} (2)-(b), we prove the following microsupport estimation of $K^\circ(\phi_{H})|_1$. We also notice that the requirement of $0_M\subset Z$ comes from the wrapping formula.
\begin{Prop}\label{prop: microsupport estimation of kernel}For a Hamiltonian function $H$ supported in $U^\infty=U/\bR_{>0}$, we have 
\[\dot{SS}(K^\circ(\phi_{H})|_1)\subset (-U)\times U .\]
In particular, we have
\[\dot{SS}(K_U)\subset (-\overline{U})\times \overline{U} .\]
\end{Prop}
\begin{proof}
By the triangle inequality, we have that
\[\dot{SS}(K^\circ(\phi_{H})|_1)\subset  \Lambda_{{\phi}_{H,1}} \cup  \Lambda_{\id} . \]
Since $\phi_{H,1}$ is compactly supported, there exists a maximal conic open set $W\subset T^*M$ at infinity such that $T^*M\setminus U \subset W$ such that $H|_W=0$. Therefore, we have
\[\dot{SS}(K^\circ(\phi_{H})|_1)\setminus {(-W^c)\times W^c} \subset \lbrace (q,-p,q,p): (q,p)\in W\rbrace.\]

Now, we want to prove that the right hand side of the above does not in $\dot{SS}(K^\circ(\phi_{H})|_1)$. To see this, we notice that $K^\circ(\phi_{H})|_1$ acts as a quantized contact transform in the sense of \cite[Theorem 7.2.1]{KS90}. In particular, the action of $K^\circ(\phi_{H})|_1$ on microstalk can can be given by geometric action given by $(\phi_{H}^1)_*$ on sheaves on cotangent bundle. Therefore, as $H|_W=0$, the monotonicity morphism $1_{\Delta_M}\rightarrow K(\phi_{H})|_1$ induces the identity map on microstalks at $(q,-p,q,p)$ for $(q,p)\in W$.

Then we have that the microstalk of $ K^\circ(\phi_{H})|_1$ at the $(q,-p,q,p)$ is zero for $(q,p)\in W$. This implies that
\[\dot{SS}(K^\circ(\phi_{H})|_1 ) \subset (-W^c)\times W^c \subset (-U)\times U.\]

The second statement follows from \autoref{coro: colimit estimation} and the wrapping formula \eqref{equation: wrapping formula}.
\end{proof}

Using the wrapping formula, we can also prove the following microsupport estimation
\begin{Prop}\label{prop: precise estimation for L_Z}If $Z$ contains the zero section, then for any increasing sequence of positive compactly supported Hamiltonians $H_n$ supported on $S^*M \setminus Z^\infty$ such that $H_n \uparrow \infty$ pointwise in $S^*M \setminus Z^\infty$, we have
\begin{align*}
  & {SS}(L_Z(F)) \subset \lbrace x : \forall n\in \bN,\,\exists {x_n} \in \dot{SS}(F)\text{, such that }\phi_{H_n,1}(x_n)\rightarrow x\rbrace \cup 0_M,\\
  & {SS}(R_Z(F)) \subset \lbrace x : \forall n\in \bN,\,\exists {x_n} \in \dot{SS}(F)\text{, such that }\phi_{-H_n,1}(x_n)\rightarrow x\rbrace \cup 0_M.
\end{align*}
\end{Prop}
\begin{proof} By \cite[Equation (1.12)]{GKS2012}, we have the microsupport estimation \cite[Equation (4.4)]{GKS2012}:
\begin{equation}\label{equation: GKS action estimation.}
\dot{SS}( K({\phi})|_{z} \circ F )= \phi_z(\dot{SS}(F)) .\end{equation}

For $L_Z$, the estimation follows immediately from \autoref{coro: colimit estimation} and \eeqref{equation: GKS action estimation.} using the wrapping formula \autoref{chris thesis}. 

For the statement for $R_Z$, we can not use \autoref{coro: colimit estimation} directly. Notice that the proof of \autoref{coro: colimit estimation} only need the fact: By abstract reasons, for all strictly increasing sequences $\{n_i\}_{i\in \bN} $, we have $\varinjlim_i F_{n_i} \simeq \varinjlim_n F_{n}$. In general, for sequences $\{n_i\}_{i\in \bN} $ as before, we do not have $\varprojlim_i F_{n_i}  \not\simeq \varprojlim_n F_{n}$ by abstract reasons. Nevertheless, for the given $H_n$'s, functions $\{H_{n_i}\}_{i\in \bN}$ still go to infinity on $S^*M \setminus Z^\infty$ for any strictly increasing subsequence $\{n_i\}_{i\in \bN} $. Therefore, \autoref{chris thesis} implies, for all strictly increasing sequences $\{n_i\}_{i\in \bN} $, that
\[R_Z(F)\simeq \varprojlim_i  \Phi_{K(\phi_{-H_{n_i}})|_1 }( F).\]
Then we can adapt the argument of \autoref{coro: colimit estimation} for the specific limit to show that
\[{SS}(R_Z(F)) \subset \liminf_n SS(\Phi_{K(\phi_{-H_n})|_1 }( F)).\]
The required microsupport estimation follows from \eeqref{equation: GKS action estimation.}. 
\end{proof}
\begin{RMK}Suppose we did not know the formula for the adjoint functor $L_Z(F)$. The proposition still shows a singular support estimation on $\varinjlim \Phi_{K(\phi_{H_n})|_1 }( F)$. It is possible to use this singular support estimation to show that $\varinjlim \Phi_{K(\phi_{H_n})|_1 }( F)\subset Z$. Subsequently, one can conclude $\varinjlim \Phi_{K(\phi_{H_n})|_1 } $ is indeed the left adjoint $L_Z$. This offers us a new proof of the wrapping formula. We leave the details to the readers. 
\end{RMK}
Last, we write down the following limit formula.
\begin{Prop}For an increasing sequence of conic open sets $U_n$ with $U=\cup_n U_n$. We set $Z_n=T^*M\setminus U_n$. Then there exists a colimit of fiber sequences in $\Sh(M\times M)$:
\[K_U\rightarrow 1_{\Delta_M}\rightarrow K_Z \simeq \varinjlim [K_{U_n}\rightarrow 1_{\Delta_M}\rightarrow K_{Z_n}].\]\    
\end{Prop}
\begin{proof}In \cite[Appendix]{CyclicZHANG}, where we prove a version of the proposition for triangulated Tamarkin categories. The construction therein relies on \cite[Proposition 2.4]{Capacities2021}, where we can prove directly here: Let $X\subset Y \subset T^*M$ two conic closed sets, we consider the inclusions $\iota_{XY}:\Sh_{X}(M) \rightarrow \Sh_{Y}(M)$. In particular, we denote $\iota_{X}=\iota_{X\,T^*M}$. The same argument of \autoref{prop: refletive} shows that $\iota_{XY}$ has a left adjoint $\iota_{XY}^*$. Moreover, as $\iota_{X}=\iota_{Y}\iota_{XY}$ by definition, we have $\iota_{X}^*=\iota_{XY}^*\iota_{Y}^*$. Therefore, the unit of $(\iota_{XY}^*, \iota_{XY})$ constructs a natural transform
\[L_X=\iota_{X}\iota_{X}^* = \iota_{Y}\iota_{XY}\iota_{XY}^*\iota_{Y}^* \leftarrow \iota_{Y}\id_{\Sh_Y(M)}\iota_{Y}^*=L_Y.\]
Similarly, we can construct a morphism $L_{T^*M\setminus Y } \rightarrow L_{T^*M\setminus X}$, and a morphism of fiber sequence 
\[\begin{tikzcd}
L_{T^*M\setminus Y } \arrow[r] \arrow[d] & \id_{\Sh(M)} \arrow[r] \arrow[d, "="] & L_Y \arrow[d] \\
L_{T^*M\setminus X} \arrow[r]            & \id_{\Sh(M)} \arrow[r]                & L_X  .        
\end{tikzcd}\]
Replacing functors in the commutative diagram by their kernels, we obtain \cite[Proposition 2.4]{Capacities2021} in the $\infty$-categorical setting.

Next, we consider the decreasing sequence $\{Z_n\}_n$ and we set $Z=Z_\infty$. For any $1\leq n<m\leq \infty$, it is clear that $\iota_{Z_m Z_n}=\iota_{Z_m Z_{m-1}}\cdots \iota_{Z_{n+1} Z_n}$. Here, we notice that for finite $m$, $n\leq n+1\leq \cdots m-1\leq m$ defines a $(n-m)$-simplex in $N(\bN)$. Therefore, taking adjoints and considering units of adjoints as above give a homotopy coherent diagram $N(\bN)\rightarrow 
 \Fun^L(\Sh(M),\Sh(M)),\,n\mapsto L_{Z_n}.$ When taking $m=\infty$, the construction gives a morphism $ \varinjlim L_{Z_n}\rightarrow L_{Z_\infty}=L_Z $. 

Lastly, we prove that the morphism $ \varinjlim L_{Z_n}\rightarrow L_Z $ is an equivalence by showing that $SS(F)\subset Z$ if and only if $ F  \xrightarrow{\simeq} \varinjlim L_{Z_n}(F)$: If $SS(F)\subset Z$, then all morphisms $ F \rightarrow  L_{Z_n}(F)$ are equivalences since $Z \subset Z_n$, so we have $ F  \xrightarrow{\simeq} \varinjlim L_{Z_n}(F)$; Conversely, if $ F  \xrightarrow{\simeq} \varinjlim L_{Z_n}(F)$, then \autoref{coro: colimit estimation} shows that 
\[SS(F) \subset \bigcap_N \overline{\bigcup_{n\geq N}SS(L_{Z_n}(F))} \subset \bigcap_N \overline{\bigcup_{n\geq N}Z_n} =\bigcap_N {Z_N} = Z.\]
 
Pass to kernels, we have a homotopy coherent diagram $N(\bN)\rightarrow \Sh(M\times M),\,n\mapsto K_{Z_n},$ and an equivalence $\varinjlim K_{Z_n}\xrightarrow{\simeq} K_Z$. For $K_U$, we notice that $K_U=\cofib{1_{\Delta_M}\rightarrow K_Z} $ and then we use colimits are commute.
\end{proof}

\section{Kashiwara-Schapira microlocal cut-off functors}\label{section: Linear microlocal cut-off functors}In this section, we study the microlocal cut-off functors defined by Kashiwara and Schapira. We will follow notation and formulation of \cite[Chapter III]{guillermou2019sheaves}.

Let $M=V$ be a real vector space of dimension $n$, and we naturally identify $T^*V=V\times V^*$. A subset $\gamma \subset V$ is a cone if $\bR_{>0}\gamma=\{xv: x>0,v\in \gamma\} \subset \gamma$, we say $\gamma$ is pointed if $0\in\gamma$. We set $\gamma^a=-\gamma$. We say a cone $\gamma$ is convex/closed if it is a convex/closed set, and is proper if $\gamma \cap \gamma^{a}=\{0\}$ (equivalently, $\gamma$ contains no line). We define the dual cone $\gamma^\circ \subset V^*$ and $\widetilde{\gamma}\subset V\times V$ as:
\[\gamma^\circ=\{l\in V^*:l(v)\geq 0, \forall v\in\gamma\},\quad \widetilde{\gamma}=\{(x,y)\in V\times V: x-y \in \gamma\}.\]

For a pointed closed cone $\gamma\subset V$, we define four functors 
\begin{equation}\label{eq:defPgam}
  \begin{alignedat}{2}
P_\gamma &: \Sh(V) \to \Sh(V), & \quad
F & \mapsto {q_{2*}}( 1_{\widetilde\gamma} \otimes {q_1^{*}} F) , \\
Q_\gamma &: \Sh(V) \to \Sh(V), & \quad
F & \mapsto {q_{2!}}( \HHOM(1_{\widetilde\gamma^a} , {q_1^!} F)) , \\
P'_\gamma &: \Sh(V) \to \Sh(V), & \quad
F & \mapsto {q_{2!}}( \HHOM(1_{\widetilde\gamma^a \setminus \Delta_V}[1], {q_1^!} F)) , \\
Q'_\gamma &: \Sh(V) \to \Sh(V), & \quad
F &\mapsto {q_{2*}}( 1_{\widetilde\gamma\setminus\Delta_V}[1] \otimes {q_1^{*}} F) .
\end{alignedat}
\end{equation}
For $\gamma = \{0\}$, we have
$\widetilde{\{0\}} = \Delta_V$ and $P_{\{0\}}(F) \simeq Q_{\{0\}}(F) \simeq F$.
Using the fiber sequence
$1_{\widetilde\gamma} \to 1_{\Delta_V} \to 1_{\widetilde\gamma \setminus
  \Delta_V}[1] $ (and the same with $\gamma^a$), we obtain the fiber sequences of functors
\begin{equation}\label{equation: fiber sequence linear cut-off}
P'_\gamma\rightarrow \id \rightarrow Q_\gamma  , \quad
P_\gamma \rightarrow \id  \rightarrow Q'_\gamma  .
\end{equation}

\begin{Prop}\label{prop: functor comparision}Let $\gamma$ be a pointed closed convex cone. For $Z=V\times \gamma^{\circ a}$ and $U=T^*M\setminus Z$, we have the isomorphisms of fiber sequences
\[ [P'_\gamma\rightarrow \id \rightarrow Q_\gamma ]\simeq [L_U\rightarrow \id\rightarrow L_Z],\quad [P_\gamma \rightarrow \id  \rightarrow Q'_\gamma ]
\simeq  [R_Z\rightarrow \id\rightarrow R_U].  
\]
\end{Prop}
\begin{RMK}Originally, \cite{KS90,guillermou2019sheaves} define triangulated functors $P_\gamma, Q_\gamma,P'_\gamma, Q'_\gamma$. Here, we use the same formula to define those functors on $\infty$-level, so they automatically descent to corresponding triangulated functors.
\end{RMK}
\begin{proof}
    It is explained in \cite[Remark III.1.9]{guillermou2019sheaves} that, for the inclusion $\iota: \Sh_{V\times \gamma^{\circ a}}(V)\rightarrow\Sh(V)$, we have that $Q_\gamma=\iota \iota^{*}$ and $P_\gamma=\iota \iota^{!}$; and the natural transformations induced by $1_{\Delta_V}\to 1_{\widetilde\gamma \setminus
  \Delta_V}[1] $ and $1_{\widetilde\gamma^a} \to 1_{\Delta_V}$ are unit/counit of the corresponding adjunctions.
\end{proof}

By the virtue of \autoref{section: verdier seq from microlocalization}, the existence of the cut-off functor appears as a purely categorical result. However, the advantage of \eeqref{eq:defPgam} is that the kernel is explicitly written (also no limit!). It is explained in \cite[Equation (III.1.5)]{guillermou2019sheaves} one can write down kernels of $P'_\gamma$ and $Q_\gamma$ as \autoref{Kernel-functor corresponding} predicted, such that the fiber sequence $K_U\rightarrow 1_{\Delta_V}\rightarrow K_Z$ is given by the fiber sequence
\[K_U=\operatorname{D}'(1_{\widetilde\gamma^a \setminus
  \Delta_V})[n-1]\rightarrow 1_{\Delta_V}\rightarrow K_Z=\operatorname{D}'(1_{\widetilde\gamma^a })[n],\]
where $\operatorname{D}'(F)=\HHOM(F,1_{V^2})$.

\begin{eg}\label{example: Tamarkin projectors}We consider a variant. Let $M=V$ be a real vector spaces. Take a pointed closed convex proper cone $\gamma \subset V$, set $U=V\times \operatorname{Int}(\gamma^\circ)$ and $Z=V \times (V^*\setminus \operatorname{Int}(\gamma^\circ))$. The microlocal cut-off functors for $U$ are studied by Tamarkin \cite{tamarkin2013} and Guillermou-Schapira \cite{GS2014} (where $L_U ,\, R_U$ are called $L_\gamma,\,R_\gamma$ respectively).

The fiber sequence of kernels is given by
\[K_U=1_{\widetilde\gamma }\rightarrow 1_{\Delta_V}\rightarrow K_V=1_{\widetilde\gamma\setminus \Delta_V }[1].\]

By \autoref{prop: functor comparision}, if we assume there exists a pointed closed convex cone $\lambda$ such that $V^*\setminus \operatorname{Int}(\gamma^\circ)=\lambda^{\circ a}$, we have $L_U=P'_\lambda=L_\gamma,  R_{U}=Q'_\lambda=R_\gamma.$
Noticed that we can show by either computation of dual functor $\operatorname{D}'$ or abstract uniqueness for the kernel that $K_U=1_{\widetilde\gamma }\simeq \operatorname{D}'(1_{\widetilde\lambda^a \setminus
  \Delta_V})[n-1]$. 
\end{eg}

\section{Results for pairs}\label{section: pair}
In this section, we consider the following construction: Take two conic closed sets $Z\subset X\subset T^*M$. Let $U=T^*M\setminus Z$ and $V=T^*M\setminus X$, then $V\subset U$. Recall that $\Sh_{X\cap U}(M;U) = \frac{\Sh_{X}(M)}{\Sh_{Z}(M)}$. 

We consider the following 9-diagram of $\bfk$-linear categories:
\begin{equation}\label{equation: 9-diagram of quotients}
\begin{tikzcd}
\Sh_{Z}(M) \arrow[r, hook] \arrow[equal]{d} & \Sh_{X}(M) \arrow[d, hook] \arrow[r, two heads] & \Sh_{X\cap U}(M;U) \arrow[d, hook] \\
\Sh_{Z}(M) \arrow[r, hook] \arrow[d, two heads]           & \Sh(M) \arrow[d, two heads] \arrow[r, two heads]              & \Sh(M;U) \arrow[d, two heads]              \\
0 \arrow[r, hook]                                                         & \Sh(M;V) \arrow[r, "\simeq "]                     & {\Sh(M;U) \Big/ \Sh_{X\cap U}(M;U) }                                                              \end{tikzcd}
\end{equation}
In this diagram, by the same argument as in \autoref{prop: refletive}, we have that all vertical and horizontal sequences are split Verdier sequences in $\PrSt(\bfk)$.  We say a morphism $f:F\rightarrow G$ in $\Sh(M;U)$ is an isomorphism on an open set $V\subset U$ if $f$ is an isomorphism in $\Sh(M;V)$. Because $\Sh_{X\cap U}(M;U)$ is a thick subcategory of $\Sh(M;U)$, we know that $f$ is an isomorphism on $V$ if and only if  $SS_U(\cofib{f})\cap V =\varnothing $.  

Here, we write down adjoint functors in the last column precisely
 \begin{equation}\label{equation: recollection diagram;pair}
    \begin{tikzcd}
\Sh_{X\cap U}(M;V)  \arrow[r, "\iota"] & \Sh(M;U)  \arrow[r, "j"] \arrow[l, "\iota^*"', bend right=30] \arrow[l, "\iota^!", bend left=30] & {\Sh(M;V).} \arrow[l, "j_!"', bend right=30] \arrow[l, "j_*", bend left=30]
\end{tikzcd}
\end{equation}
Then we have adjunction pairs
\[L_{(U,V)}\coloneqq j_!j \dashv j_*j \eqqcolon R_{(U,V)},\qquad L_{(Z,X)}\coloneqq \iota\iota^* \dashv \iota\iota^! \eqqcolon R_{(Z,X)},\]
and the units/counits give us following fiber sequences of functors on $\Sh(M;U)$
\begin{align}\label{equation: fiber sequence non-linear cut-off;pair}
        L_{(U,V)}\rightarrow \id\rightarrow L_{(Z,X)},\quad 
        R_{(Z,X)}\rightarrow \id\rightarrow R_{(U,V)}.
\end{align}
\begin{RMK}\label{remark: existence of kernel for pairs}By \eeqref{equation: kernel-functor-microlocalization}, we can also represent functors in \eeqref{equation: fiber sequence non-linear cut-off;pair} by integral kernels $K_{(Z,X)}, K_{(U,V)}\in \Sh(M\times M;(-U)\times U)$ and they are images of $K_X,K_V$ under the quotient functor $\Sh(M\times M) \rightarrow \Sh(M\times M;(-U)\times U)$. The relative version of microsupport estimation \autoref{prop: microsupport estimation of kernel} is also true.
\end{RMK}

Therefore, we can state the following version of cut-off lemma
\begin{Thm}\label{prop: non-linear microlocal cut-off lemma of pair}For conic closed sets $Z\subset X\subset T^*M$, and $U=T^*M\setminus Z$ and $V=T^*M\setminus X$, we have
\begin{enumerate}
\item 
\begin{enumerate}
\item The morphisms $F\rightarrow L_{(Z,X)}(F)$ and $R_{(Z,X)}(F)\rightarrow F$ are isomorphisms if and only if $SS_U(F)\subset X \cap U$.
\item The morphism $L_{(U,V)}(F)\rightarrow F$ is an isomorphism if and only if $F \in ^{\perp}\Sh_{X\cap U}(M;U)$, and the morphism $F\rightarrow R_{(U,V)}(F)$ is an isomorphism if and only if $F \in \Sh_{X\cap U}(M;U)^{\perp}$.
\end{enumerate}

\begin{enumerate}
 \item The morphisms $L_{(U,V)}(F)\rightarrow F$ and $F\rightarrow R_{(U,V)}(F)$ are isomorphisms on $U\cap V$.
\item If $0_M\subset Z$, the morphism $F\rightarrow L_{(Z,X)}(F)$ and $R_{(Z,X)}(F)\rightarrow F$ are isomorphisms on $\operatorname{Int}(X)\cap U$. Equivalently, we have $SS_U(L_{(U,V)}(F))\cup SS_U(R_{(U,V)}(F)) \subset \overline{V} \cap U $. Here $\overline{V}$ denotes the closure of $V$ in $T^*M$.
\end{enumerate}
\end{enumerate}
\end{Thm}
\begin{proof}Here, we notice that $SS_U([F])=SS(R_U(F))\cap U=SS(L_U(F))\cap U$ for $[F]\in \Sh(M;U)$. Then we apply the absolute version \autoref{prop: non-linear microlocal cut-off lemma} to conclude the result.    
\end{proof}

\begin{eg}\label{example: Tamarkin category}We consider the manifold $M\times \bR_t$. We take a closed set $ \mathcal{Z} \subset J^1M=T^*M\times \bR$ and $\mathcal{U}=J^1M\setminus \mathcal{Z}$. We consider their conification in $T^*(M\times \bR_t)=T^*M\times \bR_t\times \bR_\tau$. We set
\begin{align}
    \begin{aligned}
        &\Omega_{>0}=\{\tau >0\},\qquad \Omega_{-}=\{\tau \leq 0\}.\\
        &U=\{(q,\tau p,t,\tau): (q,p,t)\in \mathcal{U}, \tau > 0\},\\
       &Z_{>0}=\{(q,\tau p,t,\tau): (q,p,t)\in \mathcal{Z}, \tau > 0\},\\
        &Z=Z_{>0}\cup \Omega_-.
    \end{aligned}
\end{align}
Then $\Omega_{-}\subset Z$ are two closed conic sets. In this case, the category 
\[\Sh(M\times \bR_t;U)\simeq \Sh(M\times \bR_t; \Omega_{{>0}}) / {\Sh_{Z_{>0}}(M\times \bR_t;\Omega_{{>0}})}\]
is useful for contact topology of $J^1M$, since $\bR_{>0}$ acts on $\Omega_{>0}$ freely, with the quotient $J^1M=T^*M\times \bR_t$.

If we take $\mathcal{U}=U_0 \times\bR$ for an open set $U_0\subset T^*M$ and $\mathcal{Z}=Z_0 \times\bR$ for the closed set $Z_0=T^*M\setminus U_0 $, the category $\Sh(M\times \bR_t;U)$ is exactly the so-called Tamarkin category $\sT(U_0)$ (cf. \cite{Hochschild-Kuo-Shende-Zhang}). Tamarkin categories are first defined in \cite{tamarkin2013}, and then studied in \cite{NicolasVicherythesis, chiu2017,Ike_2019, asano2017persistence,zhang2020quantitative,asanoike2022complete,GV2022viterboconj,gammasupport-as-micro-support,Shadowdistance-Asano-Ike-Li}. In this case, $U$ admits an $\bR$-action by translation along $\bR_t$. Therefore, there exists a $\sT\coloneqq \sT(\pt)$ (which is a symmetric monoidal category \cite{GS2014,Hochschild-Kuo-Shende-Zhang}) action on $\sT(U_0)$ that helps us understand the action filtration from symplectic geometry. The functors 
\[L_{U_0}^{\sT}\coloneqq L_{(\Omega_{>0},U)},\, L_{Z_0}^{\sT}\coloneqq L_{(Z,\Omega_{-})}\]
are actually $\sT$-linear. By the $\sT$-linear left adjoint functors classification (see \cite[Proposition 5.12]{Hochschild-Kuo-Shende-Zhang}, which is an enriched version of \autoref{remark: existence of kernel for pairs}), there is a fiber sequences $K_{U_0}^{\sT} \rightarrow 1^{\sT}_{\Delta_M} \rightarrow K^{\sT}_{Z_0}$ in $\sT(T^*(M\times M))=\Sh(M\times M;\sT)$ with 
\[L_{U_0}^{\sT}=\Phi^{\sT}_{K_{U_0}^{\sT}},\,L_{Z_0}^{\sT}=\Phi^{\sT}_{K_{Z_0}^{\sT}},\]
where $\Phi^{\sT}$ means the convolution functor defined using $\sT$-linear 6-operators. These functors are studied in \cite{mc-Tamarkin,chiu2017,Capacities2021,CyclicZHANG,Hochschild-Kuo-Shende-Zhang}, where corresponding kernels are denoted by $K_{U_0}^{\sT}=P_{U_0}, K_{Z_0}^{\sT}=Q_{U_0}$. 

Here, we can apply \autoref{prop: non-linear microlocal cut-off lemma of pair} to Tamarkin categories, regarded as $\bfk$-linear categories $\Sh(M\times \bR_t;U)$, to obtain an microsupport estimation $SS_{\Omega_{>0}}(L_{(\Omega_{>0},U)}(F))\subset \overline{U}\cap \Omega_{>0}$. Then we obtain an estimation for reduced microsupport: 
\[RS(L_{U_0}^{\sT}(F))=RS(L_{(\Omega_{>0},U)}(F))=\{(q,p): (q,p,t,1)\in SS_{\Omega_{>0}}(L_{(\Omega_{>0},U)}(F))\} \subset \overline{U_0}.\]

In general, if the open set $\mathcal{U}\subset J^1M$ admits a $\mathbb{G}$-action for a subgroup $\mathbb{G}\subset \bR$ via translation, Asano, Ike and Kuwagaki introduce an equivariant version of categories, see \cite{asano2020sheaf,Kuwagaki_WKB,IK-NovikovTamarkinCategory}: We consider $\mathbb{G}$ as a discrete group by forgetting its topology, in this case, the group homomorphism $\mathbb{G} \rightarrow \bR_t$ is still continuous. Then we consider the $\mathbb{G}$-action on $\Sh(M\times \bR;\Omega_+)$ via translation, and define 
\[\Sh^{\mathbb{G}}(M\times \bR;\Omega_{>0}) \coloneqq (\Sh(M\times \bR;\Omega_{>0}))^{\mathbb{G}} \coloneqq \varprojlim_{\mathbb{G}}\Sh(M\times \bR;\Omega_{>0}).\]
There exists a forgetful functor $\mathfrak{f}:\Sh^{\mathbb{G}}(M\times \bR;\Omega_{>0}) \rightarrow \Sh(M\times \bR;\Omega_{>0})$, and the notion $SS_{\Omega_>0}(F)\coloneqq SS_{\Omega_>0}(\mathfrak{f}(F))$ is well-defined. Therefore, we can define categories 
\[\Sh^{\mathbb{G}}(M\times \bR_t;U), \quad {\Sh_{Z_{>0}}^{\mathbb{G}}(M\times \bR_t;\Omega_{{>0}})}.\]
The cut-off functor and cut-off lemma are still true and can be proven by the following ingredients: The forgetful functor $\mathfrak{f}$ admits both left and right adjoint, and $\Sh(N\times M\times \bR;\Omega_{>0}) $ acts on $\Sh^{\mathbb{G}}(M\times \bR;\Omega_{>0}) $ via a convolution and $\mathfrak{f}^L$. Then we can construct the wrapping formula via the action and show all results of this article directly and carefully (especially on the categorical aspect rather than the microlocal aspect). We also leave the details to the readers as it a little more beyond the purpose of this article.
\end{eg}

\section{K\"unneth formula}\label{section: functor classification} 
Here, we present some computations of Lurie tensor products, which generalize the result of \cite[Theorem 1.2]{Kuo-Li-Duality2024}, where an isotropic condition is needed. We take manifolds $N, M$, conic closed sets $Z\subset T^*M$ and $X\subset T^*N$, and we set $U=T^*M\setminus Z$ and $V=T^*N\setminus X$. 

To start with, we recall that $\cC\in \PrSt$ is dualizable if it is a dualizable object with respect to the symmetric monoidal structure defined by the Lurie tensor product, and its dual is denoted by $\cC^\vee$. Then we recall that:  
\begin{Lemma}[{\cite[Remark 3.7]{Hochschild-Kuo-Shende-Zhang}}]The category $\Sh_Z(M)$ is dualizable with dual $\Sh_{-Z}(M)$. The category $\Sh(M;U)$ is dualizable with the dual $\Sh(M;-U)$.     
\end{Lemma}
\begin{RMK}In fact, \autoref{prop: refletive} is enough to guarantee that $\Sh_Z(M)$ and $\Sh(M;U)$ are dualizable as stable linear categories (c.f. \cite[Proposition 1.17.]{Efimov-K-theory}), and local rigidity of $\bfk$ ensure that the dual is actually a $\bfk$-linear dual \cite[Proposition 4.3.3]{Fake_sheaves_on_mfd}. The natural of the lemma here is that we identify the dual of $\Sh_Z(M)$ and $\Sh(M;U)$.    
\end{RMK}

\begin{RMK}In the following propositions, the box tensor $\boxtimes$ of kernels should be understood after identifying $(N\times M)^2\simeq N^2 \times M^2$ via $((n_1,m_1),(n_2,m_2)) \mapsto (n_1,n_2,m_1,m_2)$.    
\end{RMK}

\begin{Thm}We have that $K_X\boxtimes 
 K_Z \simeq K_{X\times Z}$, and $\Sh_X(N)\otimes \Sh_Z(M)\simeq \Sh_{X\times Z}(N\times M)$.    
\end{Thm}
\begin{proof}We first assume that both $X$ and $Z$ containing the zero section.

For the first statement. We set $K\coloneqq K_X\boxtimes 
 K_Z $. It is clear that the morphism $1_{\Delta_N}\boxtimes 1_{\Delta_M}=1_{\Delta_{N\times M}} \rightarrow K$ induces an equivalence $\Phi_{K}\xrightarrow{\simeq} \Phi_{K}\circ \Phi_{K}  $. Henceforth, $\Phi_{K}$ is a projector in the sense of \cite[Definition 4.1.1]{KS2006}, whose essential image are given by the full subcategory spanned $F$ with $F\xrightarrow{\simeq} \Phi_{K}(F)$ (c.f. \cite[Proposition 4.1.3]{KS2006}). 
 
We are going to prove that $\Phi_{K}=\Phi_{K_{X\times Z}}$ now. Because $\Phi_{K_{X\times Z}}$ is the projector whose essential image is $\{F: F\xrightarrow{\simeq} \Phi_{K_{X\times Z}}(F)\}=\Sh_{X\times Z}(N\times M)$ by its definition. It remains to show that $F\xrightarrow{\simeq} \Phi_{K}(F)$ if and only if $ F\in \Sh_{X\times Z}(N\times M)$. 

For $F\in \Sh(N\times M)$ with $\Phi_{K}( F)\simeq F$, we can write $F=\varinjlim   F_\alpha \boxtimes G_\alpha $, and then we have $F \simeq \varinjlim \Phi_{K_X} (F_\alpha) \boxtimes \Phi_{K_Z} ( G_\alpha)$. So $SS(F)= SS(\varinjlim \Phi_{K_X} ( F_\alpha) \boxtimes \Phi_{K_Z} ( G_\alpha))\subset \overline{SS(\Phi_{K_X} ( F_\alpha) ) \times SS(\Phi_{K_Z} ( G_\alpha))} \subset X\times Z $.

Conversely, we need to show that $K$ fixes $F\in \Sh_{X\times Z}(N\times M)$. We use the wrapping formula \autoref{chris thesis}, and adapt the argument of \cite[Lemma 4]{gammasupport-as-micro-support}. For cofinal sequences of conic non-negative Hamiltonian functions $H_{\lambda_1(n_1)}$ supported in $U$ and $H_{\lambda_2(n_2)}$ supported in $V$ with index sequences $\lambda_i: \bN \rightarrow \bN$, we have that $K\simeq \varinjlim_{(n_1,n_2)} K(\phi_{H_{\lambda_1(n_1)}})|_1\boxtimes  K(\phi_{H_{\lambda_2(n_2)}})|_1$. Now, consider the 1-parameter version of GKS quantization $K(\phi_{H_{\lambda_1(n_1)}})$ and $K(\phi_{H_{\lambda_2}(n_2)})$ as given in \autoref{theorem: GKS}. By a parameter version of \eqref{equation: GKS action estimation.} (precisely, \cite[Equation (4.3)]{GKS2012}), it is direct to see that for $W=\Phi_{K(\phi_{H_{\lambda_1(n_1)}})\boxtimes K(\phi_{H_{\lambda_2(n_2)}})} (F)\in \Sh(N\times \bR\times M\times \bR)$, we have
\begin{equation*}
 SS(W)\subset 
 \left\lbrace (q_1,z_1, p_1, \zeta_1,q_2 ,z_2,p_2,\zeta_2): \begin{aligned} &\exists (q_1',p_1',q_2',p_2')\in SS(F) ,\\ 
                 &(q_i,p_i)=\phi_{z_i}^{H_{\lambda_i(n_i)}}(q_i',p_i'),\\
                 &\zeta_i=-H_{\lambda_i(n_i)}(q_i',p_i')  \end{aligned}        \right\rbrace .
\end{equation*}
Now, since $SS(F)\subset X\times Z$, we have that $SS(W)\subset T^*(N\times M) \times 0_{\bR^2}$. Then by \cite[Proposition 5.4.5]{KS90}, we have $F=W|_{z_1=z_2=0}\simeq W|_{z_1=z_2=1}= \Phi_{K(\phi_{H_{\lambda_1(n_1)}})|_1\boxtimes  K(\phi_{H_{\lambda_2(n_2)}})|_1} ( F)$. The isomorphism is commute with continuation map given by non-negativity of $H_{\lambda_i(n_i)}$ (see \autoref{theorem: GKS}). Then we taking limit with respect to $(n_1,n_2)$ to see that $\Phi_{K}( F)\simeq F$.

Consequently, we have $K_{X\times Z}\simeq K=K_X\boxtimes 
 K_Z$ by uniqueness of kernel.

For the second statement, under the natural identification $\Sh(N)\otimes \Sh(M)=\Sh(N\times M)$ (see \cite[Proposition 2.30]{6functor-infinity}), it remains to verify that the essential image of the functor\[\Sh_X(N)\otimes\Sh_{Z}( M)\rightarrow \Sh(N)\otimes\Sh( M) = \Sh(N\times M), \quad F\otimes G \mapsto F\boxtimes G\] is $\Sh_{X\times Z}(N\times M)$. In fact, any object 
$F$ in $\Sh_{X\times Z}(N\times M)$ can be written
as a colimit of some $F=\varinjlim F_\alpha \boxtimes G_\alpha \in \Sh(N\times M)$. We should show that the $F_\alpha$ can be chosen in $\Sh_X(N)$ and $G_\alpha$ can be chosen to be in $\Sh_Z(M)$. By definition of $K_{X\times Z}$, we have $F\simeq \Phi_{K_{X\times Z}}(F)$. As we have already proved that $K_{X\times Z}\simeq K_X\boxtimes K_Z$, we have
\[F  \simeq 
\Phi_{K_{X\times Z}}(\varinjlim F_\alpha \boxtimes G_\alpha) \simeq \varinjlim \Phi_{K_X\boxtimes K_Z}(F_\alpha \boxtimes  G_\alpha  )=
\varinjlim \Phi_{K_X}(F_\alpha) \boxtimes  \Phi_{K_Z}(G_\alpha  ).
\]
Then by replacing $F_\alpha$ and $G_\alpha$ by $\Phi_{K_X}(F_\alpha)$ and $\Phi_{K_Z}(G_\alpha  )$, we can conclude
that $F\in \Sh_X(N)\otimes\Sh_{Z}( M)$.

Now, we remove the assumption that $X,Z$ contain the zero section. In this case, we set $Z_0=Z\cap 0_M$ and $\Tilde{Z}=Z\cup 0_M$ (similarly for $X_0$ and $\Tilde{X}$). Then we use the result for $\Tilde{X}$ and $\Tilde{Z}$, which contain zero sections, from above; and conclude by noticing that $SS(F)\subset Z$ if and only if $SS(F) \subset \Tilde{Z}$ and $\supp{F}\subset Z_0$, and $F_n\otimes G_m\simeq  (F\boxtimes G)_{(n,m)}$ for $(n,m)\in N\times M$.
\end{proof}

\begin{Coro}
\label{proposition: relative case}We have an equivalence of fiber sequences of categories
\begin{equation*}
    \begin{split}
       &\Sh_{X\times Z}(N\times M) \rightarrow \Sh_{X\times T^*M }(N\times M) \rightarrow  \Sh_{X\times U }(N\times M;T^*N\times U) \\
       \simeq &\Sh_X(N)\otimes (\Sh_{ Z}( M) \rightarrow \Sh( M) \rightarrow  \Sh(M;U)).
    \end{split}
\end{equation*}

\end{Coro}
\begin{proof}The functor $- \mapsto \Sh_X(N)\otimes -$ preserves colimits and fully-faithfulness. This can be seen from two different points of view, one from dualizability of $\Sh_X(N)$ and \cite[Theorem 2.2]{Efimov-K-theory}, the other one from \cite[Corollary 2.29]{Haine-Nonabelian-basechange} since in the stable setting recollement pair and split Verdier sequence are equivalent notion \cite[Proposition A.2.11]{Hermitian-K-theory} (in this approach we do not need dualizability). Therefore, we have the equivalence of Verdier sequences
\begin{equation*}
    \begin{split}
      &\Sh_X(N)\otimes (\Sh_{ Z}( M) \rightarrow \Sh( M) \rightarrow  \Sh(M;U)) \\
       \simeq &\Sh_X(N)\otimes\Sh_{ Z}( M) \rightarrow \Sh_X(N)\otimes\Sh( M) \rightarrow  \Sh_X(N)\otimes\Sh(M;U).\qedhere
    \end{split}
\end{equation*}
\end{proof}

\begin{Thm}We have an equivalence 
\[\Sh(N;V)\otimes \Sh(M;U) \simeq \Sh(N\times M;V\times U).\] 
In particular, we have that $K_V\boxtimes 
 K_U \simeq K_{V\times U}$.
\end{Thm}
\begin{proof}We fill the 9-diagram in \eeqref{equation: 9-diagram of quotients} by setting the manifold by $N\times M$, two conic closed sets are $T^*N\times Z \subset T^*N\times Z\cup X\times T^*M$. Therefore, we have the following equivalence
\[\Sh(N\times M;V\times U)\simeq \Sh(N\times M;T^*M\times U)/\Sh_{X\times U}(N\times M;T^*M\times U).\]
Now, we apply \autoref{proposition: relative case} twice to see that
\[[\Sh_{X\times U}(N\times M;T^*N\times U) \hookrightarrow \Sh(N\times M;T^*N\times U)] = [\Sh_X(N)\hookrightarrow \Sh(N)]\otimes \Sh(M;U).\]
Consequently, we have
\[\Sh(N\times M;T^*N\times U)/\Sh_{X\times U}(N\times M;T^*N\times U)\simeq  \Sh(N;V)\otimes \Sh(M;U).\qedhere\]
\end{proof}

\bibliographystyle{bingyu}
\clearpage
\phantomsection
\bibliography{rsump}

\noindent
\parbox[t]{28em}
{\scriptsize{
\noindent
Bingyu Zhang\\
Centre for Quantum Mathematics, University of Southern Denmark\\
Campusvej 55, 5230 Odense, Denmark\\
Email: {bingyuzhang@imada.sdu.dk}
}}

\end{document}